\documentclass[12pt, a4paper]{amsart}
\usepackage{amsmath, amssymb,amsthm}
\usepackage{lettrine}
\usepackage[Sonny]{fncychap}
\usepackage[english]{babel}
\usepackage{amsmath, amssymb,amsthm}
\usepackage[all]{xy}
\usepackage{tabularx}
\usepackage[applemac]{inputenc}
\usepackage{graphicx} 
\usepackage[dvipsnames]{xcolor}
\newtheorem{theorem}{Theorem}[section]
 
\newtheorem*{thm-num}{Theorem} 
\newtheorem{prop}[theorem]{Proposition}
\newtheorem{lemma}[theorem]{Lemma}
\newtheorem{cor}[theorem]{Corollary}

\theoremstyle{definition}
\newtheorem{rem}{Remark}

\newcommand{\C}{\mathbf{C}}
\newcommand{\Q}{\mathbf{Q}}

\newcommand{\Z}{\mathbf{Z}}

\newcommand{\cf}{\textit{cf. }}

\newcommand{\cO}{\mathcal{O}}

\newcommand{\cT}{\mathcal{T}}

\DeclareMathOperator{\Frob}{Frob}

\DeclareMathOperator{\SL}{SL}
\DeclareMathOperator{\Gal}{Gal}

\DeclareMathOperator{\PGL}{PGL}
\DeclareMathOperator{\rH}{H}

\DeclareMathOperator{\Spec}{Spec}

\title{Congruence formulae for Legendre modular polynomials}
\author{Adel Betina and Emmanuel Lecouturier}

\address{Universitat Polit\`ecnica de Catalunya} 
\email{adel.betina@upc.edu}  
\address{Universit\'e Paris 7}
\email{emmanuel.lecouturier@imj-prg.fr}  

\begin{document}

\maketitle
\begin{abstract}
Let $p\geq 5$ be a prime number. We generalize the results of E. de Shalit \cite{shalit2} about supersingular $j$-invariants in characteristic $p$. 

We consider supersingular elliptic curves with a basis of $2$-torsion over $\overline{\mathbf{F}}_p$, or equivalently supersingular Legendre $\lambda$-invariants. Let $F_p(X,Y) \in \mathbf{Z}[X,Y]$ be the $p$-th modular polynomial for $\lambda$-invariants. A simple generalization of Kronecker's classical congruence shows that $R(X):=\frac{F_p(X,X^{p})}{p}$ is in $\mathbf{Z}[X]$. We give a formula for $R(\lambda)$ if $\lambda$ is a supersingular. This formula is related to the Manin--Drinfeld pairing used in the $p$-adic uniformization of the modular curve $X(\Gamma_0(p)\cap \Gamma(2))$. This pairing was computed explicitly modulo principal units in a previous work of both authors. Furthermore, if $\lambda$ is supersingular and lives in $\mathbf{F}_p$, then we also express $R(\lambda)$ in terms of a CM lift (which are showed to exist) of the Legendre elliptic curve associated to $\lambda$. \end{abstract}

\section{Introduction}

Let $p \geq 5$ be a prime number. 

We are interested in this article in the modular curve $X(\Gamma(2) \cap \Gamma_0(p))$. A plane equation of this curve is given by the classical $p$-th \textit{modular polynomial \`a la Legendre}, which we denote by $F_p(X,Y)$. It is shown that it satisfies the same properties as the classical modular polynomials for the $j$-invariants, namely it is symmetric, has integer coefficients and we have the Kronecker congruence $F_p(X,Y) \equiv (X^p-Y)(X-Y^p)$ modulo $p$. 

This last congruence can be roughly interpreted by saying that the reduction modulo $p$ of $X(\Gamma(2) \cap \Gamma_0(2))$ is a union of two irreducible components isomorphic to $\mathbf{P}^1$. In this work, we show a congruence formula for $F_p(X,X)$ modulo $p^2$, which intuitively gives us information about the reduction of our curve modulo $p^2$. The tools that we use to study this reduction is the $p$-adic uniformization (due to Mumford and Manin--Drinfeld). This was already used by E. de Shalit in his paper \cite{shalit2}, and we follow his method in our case. We do a deep study of some annuli in the supersingular residue disks of the rigid modular curve. 

By combining our previous work (\cf \cite{B-L}) on the $p$-adic uniformization of our modular curve and the present results, we obtain an elementary formula for values taken by $F_p(X,X)$ modulo $p^2$ (which does not however give us a formula for the polynomial itself). We now give more precise details about ours results.

Let $\mathfrak{M}_{\Gamma_0(p) \cap \Gamma(2)}$ be the stack over $\Z[1/2]$ whose $S$-points are the isomorphism classes of generalized elliptic curves $E/S$, endowed with a locally free subgroup $A$ of rank $p$ such that $A+E[2]$ meets each irreducible component of any geometric fiber of $E$ ($E[2]$ is the subgroup of $2$-torsion points of $E$) and a basis of the $2$-torsion (\textit{i.e.} an isomorphism $\alpha_2: E[2] \simeq (\Z/2\Z)^2$). Deligne and Rapoport proved in \cite{D-R} that $\mathfrak{M}_{\Gamma_0(p) \cap \Gamma(2)}$ is a regular algebraic stack, proper, of pure dimension $2$ and flat over $\Z[1/2]$. 

Let $M_{\Gamma_0(p) \cap \Gamma(2)} $ be the coarse space of the algebraic stack $\mathfrak{M}_{\Gamma_0(p) \cap \Gamma(2)}$ over $\Z[1/2]$. Deligne-Rapoport proved that $M_{\Gamma_0(p) \cap \Gamma(2)} $ is a normal scheme and proper flat of relative dimension one over $\Z[1/2]$. Moreover, Deligne-Rapoport proved that $M_{\Gamma_0(p) \cap \Gamma(2)} $ is smooth over $\Z[1/2]$ outside the points associated to supersingular elliptic curves in characteristic $p$ and that $M_{\Gamma_0(p) \cap \Gamma(2)} $ is a semi-stable regular scheme (\cf \cite[V.1.14, Variante]{D-R} and \cite[Proposition 2.1]{B-L} for more details).

Let  $K$ be the unique quadratic unramified extension of $\Q_p$, $\cO_K$ be the ring of integers of $K$ and $k$ be the residual field. Let $\mathfrak{X}$ be the base change  
$M_{\Gamma_0(p) \cap \Gamma(2)} \otimes \cO_K$; it is the coarse moduli space of the base change $\mathfrak{M}_{\Gamma_0(p) \cap \Gamma(2)} \otimes \cO_K$ (because the formation of coarse moduli space commutes with flat base change).

Let $M_{\Gamma(2)}$ be the model over $\Z[1/2]$ of the modular curve $X(2)$ introduced by Igusa \cite{Igusa}. The special fiber of the scheme $\mathfrak{X}$ is the union of two copies of $M_{\Gamma(2)}{ \small \otimes k}$ meeting transversally at the supersingular points, and such that a supersingular point $x$ of the first copy is identified with the point $x^p=\Frob_p(x)$ of the second copy (the supersingular points of the special fiber of $\mathfrak{X}$ are $k$-rational). Moreover, we have $M_{\Gamma(2)}{ \small \otimes k} \simeq \mathbf{P}^1_k$ (\cf \cite[Proposition $2.1$]{B-L}).

The cusps of $M_{\Gamma_0(p) \cap \Gamma(2)}$ correspond to N\'eron $2$-gons or $2p$-gons and are given by sections $\Spec \Z[1/2] \rightarrow \mathfrak{M}_{\Gamma_0(p) \cap \Gamma(2)}$ composed with the coarse moduli map $\mathfrak{M}_{\Gamma_0(p) \cap \Gamma(2)} \rightarrow M_{\Gamma_0(p) \cap \Gamma(2)}$.
f

Mumford's theorem \cite{M} implies that the rigid space $\mathfrak{X}^{rig}$ attached to $\mathfrak{X}$ is the quotient of a $p$-adic half plane $\mathfrak{H}_\Gamma=\mathbf{P}_K^1-\mathcal{L}$ by a Schottky group $\Gamma$, where $\mathcal{L}$ is the set of the limits points of $\Gamma$. Manin and Drinfeld constructed a pairing $\Phi: \Gamma^{ab} \times \Gamma^{ab} \rightarrow K^{\times}$ in \cite{D} and explained how this pairing gives a $p$-adic uniformization of the Jacobian of $\mathfrak{X}_K$.

Let $\Delta$ be the dual graph of the special fiber of $\mathfrak{X}$. Mumford's construction shows that $\Gamma$ is isomorphic to the fundamental group $\pi_1(\Delta)$. The abelianization of $\Gamma$ is isomorphic to to the augmentation subgroup of the free $\mathbf{Z}$-module with basis the isomorphism classes of supersingular elliptic curves over $\overline{\mathbf{F}}_p$. Let $S:=\{e_i\}$ be the set of supersingular points of $\mathfrak{X}_{k}$. We proved in \cite{B-L}, using the ideas of \cite{shalit}, that the pairing $\Phi$ can be expressed, modulo the principal units, in terms of the modular invariant $\lambda$ as follow.

\begin{enumerate} 
 
\item The Manin--Drinfeld pairing $\Phi:\Gamma^{ab} \times \Gamma^{ab} \rightarrow K^{\times}$  takes values in $\Q_p^{\times}$.

\item Let $\bar{\Phi}$ be the residual pairing modulo the principal units $U_1(\Q_p)$ of $\Q_p$, then, after the identification $\Gamma^{ab} \simeq H_1(\Delta, \mathbf{Z}) \simeq \Z[S]^0$, $\bar{\Phi}$ extends to a pairing 

$\Z[S]\times \Z[S] \rightarrow K^{\times}/U_1(K)$ such that:

{\small 
 $$ \bar{\Phi}(e_i,e_j) = \left\{ 
 \begin{array}{cl}
 (\lambda(e_i)-\lambda(e_j))^{p+1} & \mbox{ if  }  i \ne j \text{;} \\
 \pm p \cdot \prod_{k \ne i} (\lambda(e_i)-\lambda(e_k))^{-(p+1)}   \mbox{ if  } i=j .   
 \end{array}
 \right. $$ }
where the sign $\pm$ is $+$ except possibly if $p\equiv 3 \text{ (modulo }4\text{)}$ and $\lambda(e_i)\not\in \mathbf{F}_p$.

\end{enumerate}

\begin{rem}\

i) We have also proved an analogue of the above result when $N=3$ and $p \equiv 1 \text{ (mod }3\text{)}$, for a suitable model $\mathfrak{X}$ of the modular curve of level $\Gamma_0(p) \cap \Gamma(3)$ over $\Z_p$.

ii) The above formula was first conjectured by Oesterl\'e using the modular invariant $j$ instead of the modular invariant $\lambda$ for the modular curve $X_0(p)$ instead of $\mathfrak{X}$, and E. de Shalit proved this conjecture in \cite{shalit} (up to a sign if $p \equiv 3 \text{ (mod }4\text{)}$).

\end{rem}

We recall that the Lambda modular invariant $\lambda : M_{\Gamma(2)} \otimes \Q \rightarrow \mathbf{P}^1_\Q$ is an isomorphism of curves. Let $F_p(X,Y) \in \mathbf{C}[X,Y]$ be the unique polynomial such that for all $\tau$ in the complex upper-half plane, we have:

$$ F_p(\lambda,X)= (X-\lambda(p \tau)) \cdot\prod_{0 \leq a \leq p-1} (X-\lambda((\tau + a)/p)).$$

Note that this polynomial has much smaller coefficients than the corresponding polynomial for the $j$-invariants.
For example, we have:
\begin{align*}
F_3(X,Y)&=X^4+X^3(-256Y^3+384Y^2-132Y )+X^2(384Y^3-762Y^2+384Y )\\&+X(-132Y^3+384Y^2-256Y )+Y^4
\end{align*}
and
\begin{align*}
F_5(X,Y)&=X^6+Y^6-65536\cdot X^5Y^5 + 163840 \cdot X^5Y^4 + 163840 \cdot X^4Y^5 - 138240 \cdot X^5Y^3 \\&- 133120\cdot X^4Y^4 - 138240 \cdot X^3Y^5  + 43520 \cdot X^5 \cdot Y^2  - 207360\cdot X^4Y^3 \\&- 207360 \cdot X^3Y^4 + 43520 \cdot X^2Y^5  - 3590\cdot X^5Y + 133135 \cdot X^4Y^2 \\& + 691180 \cdot X^3Y^3 + 133135\cdot X^2Y^4 - 3590 \cdot XY^5  + 43520 \cdot X^4 Y - 207360 \cdot X^3Y^2 \\& - 207360 \cdot X^2 Y^3 + 43520 \cdot X Y^4 - 138240 \cdot X^3Y - 133120 \cdot X^2 Y^2 - 138240 \cdot X \cdot Y^3  \\& + 163840 \cdot X^2Y + 163840 \cdot XY^2 - 65536 \cdot XY
\end{align*}
while the corresponding polynomials for $j$-invariant are enormous.

 This agrees with the philosophical principle that adding a $\Gamma(2)$ structure simplifies a lot the computations. Another instance of this principle was applied in a paper of the second author about the Eisenstein ideal and the supersingular module. Also, in the case of $\lambda$-invariants, there are no complications due to the elliptic points, so the formula are smoother and there is no conjectural sign as in the $j$ case of E. de Shalit. This principle is one of the motivation we had to generalize E. de Shalit's results to our case.

In this article, we prove that the affine scheme $\Spec \Q[X,Y]/(F_p(X,Y))$ is a plane model of $M_{\Gamma_0(p) \cap \Gamma(2)}$ over $\Q$ (i.e both curves are birational), and that the polynomial $F_p(X,Y)$ satisfies the same basic properties as Kronecker's $p$-th modular polynomial for the modular curve $X_0(p)$. We derive another formula for the diagonal values of $\bar{\Phi}$, related to the polynomial $F_p$ as follow.

\begin{theorem}\label{main_theorem}
\begin{enumerate}\ 

\item We have $F_p(\lambda,X) \in \Z[\lambda,X]$ and $F_p$ gives a plane model the modular curve $M_{\Gamma_0(p) \cap \Gamma(2)} \otimes \Q$.

\item For any lift $\beta_i$ of $\lambda(e_i)$ in $K$, we have $$\bar{\Phi}(e_i,e_i) \equiv F_p(\beta_i,\beta_i^p) \text{ (modulo }U_1(K) \text{).}$$
\item Assume that $\lambda(e_i) \in \mathbf{F}_p$. Let $E_i$ be a lift of $e_i$ to an elliptic curve over $\Q_p(\sqrt{-p})$ with complex multiplication by the maximal order of $\Q_p(\sqrt{-p})$. Then $$\bar{\Phi}(e_i,e_i) \equiv (\lambda(E_i) - \lambda(E_i)^p)^2 \text{ (modulo }U_1(\overline{\mathbf{Q}}_p) \text{).}$$

\end{enumerate}  
\end{theorem}

Our approach is based on the technics of $p$-adic uniformization of \cite{shalit} and \cite{shalit2}, on a deep analysis of the supersingular annuli in $\mathfrak{X}^{an}$ and on the action of the Atkin-Lehner involution $w_p : \mathfrak{X} \simeq \mathfrak{X}$. The key point is to relate the diagonal elements of the extended period matrix $\bar{\Phi}$ to the polynomial $F_p(X,Y)$.

\begin{cor} Let $R(X)=F_p(X,X^p)/p \in \Z[X]$, and $\bar{R}(X) \in \mathbf{F}_p$ be the reduction of $R(X) \mod p$. Let $\lambda(e_i) \in \mathbf{F}_{p^2}$ be the $\lambda$-invariant of a supersingular elliptic curve $e_i$. Then, we have: $$\bar{R}(\lambda(e_i))=\pm(-1)^{\frac{p-1}{2}}\prod_{k\ne i} (\lambda(e_i)- \lambda(e_k))^{-(p+1)}$$
where the $\pm$ sign is $+$ except possibly if $p \equiv 3 \text{ (modulo } 4\text{)}$ and $\lambda(e_i) \not\in \mathbf{F}_p$.
 On the other hand, if $\lambda \in \mathbf{F}_{p^2}$ is not a supersingular invariant, then $\bar{R}(\lambda)=0$.

\end{cor}

\begin{proof}
The first assertion follows by comparing Theorem \cite[1]{B-L} and the Theorem above. 

Let $\lambda \in \mathbf{F}_{p^2}$ which is not supersingular.
Let $\tilde{X}$ be the scheme over $ \cO_K$ defined by $F_p(X,Y)=0$. Let $\beta \in \cO_K$ a lift of $\lambda$ (this lift exists since $M_{\Gamma(2)}$ is proper over $\Z[1/2]$). The closed point $x$ of $\tilde{X}$ corresponding to the maximal ideal $\mathcal{M}=(p, X-\beta, Y-\beta^p) \subset \cO_K[X,Y]$ is regular on $\tilde{X}$ if and only if $F_p(X,Y)$ doesn't belongs to $\mathcal{M}^2$. Using Taylor expansion of $F_p$ at $(\beta,\beta^p)$, we get that $F_p(X,Y)=p.R(\beta)+ F_X(\beta,\beta^p)(X-\beta)+F_Y(\beta,\beta^p)(Y-\beta^p) \mod \mathcal{M}^2$. But it is clear from the Kronecker's congruence that $F_X(\beta,\beta^p)$ and $F_X(\beta,\beta^p)$ are divisible by $p$. Thus, our regularity conditions is equivalent to the fact that $\bar{R}(\lambda) \ne 0$. 

Corollary \ref{Kronecker_congruence} shows that $(\lambda,\lambda^p)$ is a singular point of the special fiber of $\tilde{X}$. But $E$ is ordinary, so the corresponding point on $\mathfrak{X}$ is smooth. Since, the minimal regular model of the normalization of $\tilde{X}$ is unique, it is $\mathfrak{X}$. Since any local regular ring is normal (since it is factorial), the point $x$ is not regular in $\tilde{X}$ and $\bar{R}(\lambda(E))=0$. 
\end{proof}

\section*{Notation}
\begin{enumerate}

\item For any algebraic extension $k$ of the field $\Z/p\Z$, we denote by  $\bar{k}$ the separable closure of $k$.
\item For any congruence subgroup $\Gamma$ of $\SL_2(\mathbf{Z})$, we denote by $\mathfrak{M}_\Gamma$ the stack over $\Z$ whose $S$-points classify generalized elliptic curves over $S$ with a $\Gamma$-level structure.
\item For any congruence subgroup $\Gamma$ of $\SL_2(\mathbf{Z})$ and any $c \in \mathbf{P}^1(\mathbf{Q})$, we denote by $[c]_{\Gamma}$ the cusp of $\Gamma \backslash (\mathfrak{H} \cup \mathbf{P}^1(\mathbf{Q}))$ corresponding to the class of $c$, where $\mathfrak{H}$ is the (complex) upper-half plane.

\item For two congruence subgroups $\Gamma$ and $L$, we denote by $\mathfrak{m}_{\Gamma \cap L}$ for the fiber product of algebraic stacks $\mathfrak{M}_{\Gamma} \times_{\mathfrak{M}} \mathfrak{M}_L$, where $\mathfrak{M}$ is the stack over $\Z$ whose $S$-points classify generalized elliptic curves over the scheme $S$.

\item For any algebraic stack $\mathfrak{M}_{\Gamma}$ over $\Z$, we denote by $M_{\Gamma}$ the coarse moduli space attached to $\mathfrak{M}_{\Gamma}$ ($M_{\Gamma}$ is an algebraic space). 

\item For any proper and flat scheme $\mathfrak{X}$ over $\mathcal{O}_K$, we denote by $\mathfrak{X}_K^{an}$ the rigid analytic space given by the generic fiber of the completion of $\mathfrak{X}$ along its special fiber (\textit{i.e.} $\mathfrak{X}_K(\bar{K}) \simeq \mathfrak{X}^{an}_K(\bar{K}))$. 

\item Let $\mid . \mid_p$ be a $p$-adic valuation on $\bar{\mathbf{Q}}_p$, then we shall denote by $x \sim y$ if and only if $\mid xy^{-1}- 1\mid_p <1$. 

\end{enumerate}

\textit{Acknowledgements.}
The first author has received funding from the European Research Council (ERC) under the European Union's Horizon 2020 research and innovation programme (grant agreement No 682152).
The first named author (A.B.) would like to thank the Universit\'e Paris 7, where most of our discussions took place, for its hospitality. The second named author (E.L.) has received funding from Universit\'e Paris $7$ for his Phd thesis and would like to thank this institution.

\section{Basic properties of coarse moduli spaces of moduli stacks of generalized elliptic curves with $\Gamma(2)$-level structure}

Let $\mathfrak{M}_{\Gamma(2)}$ be the stack over $\Z[1/2]$ parametrizing generalized elliptic curves with $\Gamma(2)$-level structure (see  \cite[IV. Definition 2.4]{D-R}). Deligne-Rapoport proved in \cite[Theorem 2.7]{D-R} that $\mathfrak{M}_{\Gamma(2)}$ is an algebraic stack proper smooth and of relative dimension one over $\Z[1/2]$. Let $M_{ \Gamma(2)}$ be the coarse algebraic space associated to $\mathfrak{M}_{\Gamma(2)}$. Proposition \cite[VI.6.7]{D-R} implies that $M_{ \Gamma(2)}$ is smooth over $\Z[1/2]$; and hence $M_{ \Gamma(2)} \otimes \cO_K$ is a scheme since it is a regular algebraic space of relative dimension one over $\cO_K$.

By the universal property of the coarse moduli space attached to an algebraic stack over a noetherian scheme, we have a coarse moduli map $$g:  \mathfrak{M}_{\Gamma(2)} \rightarrow M_{ \Gamma(2)}$$ such that for any field $L$ of characteristic different from two, $g$ induces a bijection $$\mathfrak{M}_{ \Gamma(2)}(L) \simeq M_{  \Gamma(2)}(L).$$

For any elliptic curve $E$ with a basis of its $2$-torsion over a field $L$ of characteristic different from $2$, $E$ is isomorphic to a unique Legendre curve $E_{\lambda}$: $Y^2=X(X-1)(X-\lambda)$ with basis of $2$-torsion the points $(0,0)$ and $(0,1)$. Hence, we have a bijection $$M_{ \Gamma(2)}(L) \rightarrow \mathbf{P}^1_{\Z[1/2]}(L),$$

associating to an elliptic curve $E$ its lambda invariant $\lambda$.

\begin{prop}\label{iso lambda} There exists an isomorphism $\lambda : M_{\Gamma(2)} \rightarrow \mathbf{P}^1_{\mathbf{Z}[1/2]}$ inducing the previous map on $L$-points for every field $L$ of characteristic different from $2$.
\end{prop}
 
\begin{proof}\

We use similar arguments to those given in the proof of \cite[VI. Theorem 1.1]{D-R}. Let $c=[1]_{\Gamma(2)}$ be the cusp such that the complex modular invariant $\lambda$ has a pole. Since a cusp of $\mathfrak{M}_{\Gamma(2)}$ is given by a section $\Spec \Z[1/2] \rightarrow \mathfrak{M}_{\Gamma(2)}$, then by composing with the coarse moduli map $g$, a cusp of $M_{\Gamma(2)}$ is also given by a section $\Spec \Z[1/2] \rightarrow M_{\Gamma(2)}$. Since $\mathfrak{M}_{\Gamma(2)}$ is proper, there exists a section $c: \Spec \Z[1/2] \rightarrow M_{\Gamma(2)}$ corresponding to $[1]_{\Gamma(2)}$ after base change to $\C$. 
 
Now, we obtain a section $c: \Spec \mathbf{Z}[1/2] \rightarrow M_{\Gamma(2)} $ giving rise to a Cartier divisor $D$ of  $\mathfrak{M}_{\Gamma(2)}$. By proposition \cite[V. 5.5]{D-R}, the geometric fibers of $M_{\Gamma(2)} $ are absolutely irreducible. The genus is constant on the geometric fibers of $M_{\Gamma(2)}$ and equals the genus of the complex modular curves $M_{\Gamma(2)}(\C)$, which is zero (see proposition \cite[7.9]{EGA}). Hence, by applying Riemann-Roch to each geometric fiber of $M_{\Gamma(2)} $ and using the base change compatibility of cohomology on geometric fibers (see \cite[III. Corollary 9.4]{H}), we obtain: $$\rH^1(M_{\Gamma(2)}, \cO_{M_{\Gamma(2)}}(D)) = 0.$$ 
 
Thus, $\rH^0(M_{\Gamma(2)}, \cO_{M_{\Gamma(2)}}(D))$ has rank two over $\mathbf{Z}[1/2]$ and is generated by $\{1,\lambda\}$, so we have a morphism $\lambda: M_{\Gamma(2)} \rightarrow \mathbf{P}^1_{\mathbf{Z}[1/2]}$ (which we normalize so that it coincides with the Legendre lambda-invariant on $L$-points as above). On each geometric fiber $\mathfrak{M}_{\Gamma(2)} \otimes {\bar{k}}$ away from characteristic $2$, we see that the degree of the divisor $D_{\bar{k}}$ corresponding to $D$ on $\mathfrak{M}_{\Gamma(2)} \otimes {\bar{k}}$ is $1$, hence $D_{\bar{k}}$ is very ample (see \cite[IV. corollary 3.2]{H}). Thus, $D$ is relatively very ample over $\mathbf{Z}[1/2]$ (see \cite[9.6.5]{EGA3}) and $\lambda$ is an isomorphism.

\end{proof}

Let $M_{\Gamma(2)}'$ be the affine open of $M_{\Gamma(2)}$ corresponding to $\mathbf{A}^1_{\mathbf{Z}[1/2]} \subset \mathbf{P}^1_{\Z[1/2]}$, $c : M_{\Gamma_0(p) \cap \Gamma(2)} \rightarrow M_{\Gamma(2)}$ be the forgetful of the $\Gamma_0(p)$-level structure (\cf \cite[IV Proposition 3.19]{D-R}), and $M_{\Gamma_0(p) \cap \Gamma(2)}'$ be the inverse image of $M_{\Gamma(2)}'$ by $c$, which is an affine scheme since $c$ is a finite morphism. Denote by $w_p$ the Atkin--Lehner involution on $M_{\Gamma_0(p) \cap \Gamma(2)}$; it preserves $M_{\Gamma_0(p) \cap \Gamma(2)}'$ by \cite[Lemma 7.4]{B-L}).
 Thus, we obtain finite maps $$ (c, c \circ w_p) : M_{\Gamma_0(p) \cap \Gamma(2)} \rightarrow M_{\Gamma(2)} \times M_{\Gamma(2)}$$
and $$ (c, c \circ w_p) : M_{\Gamma_0(p) \cap \Gamma(2)}' \rightarrow M'_{\Gamma(2)} \times M_{\Gamma(2)}'.$$
 
 Let $R$ such that $M_{\Gamma_0(p) \cap \Gamma(2)}'= \Spec R$ and $\Spec \Z[1/2][\lambda,\lambda'] = M_{\Gamma(2)} \times M_{\Gamma(2)}$. The image of the finite morphism (hence proper) $(c, c \circ w_p) : M_{\Gamma_0(p) \cap \Gamma(2)}' \rightarrow M'_{\Gamma(2)} \times M_{\Gamma(2)}'$ is a reduced closed subset $V(I)$, where $I$ is an ideal of $\Spec \Z[1/2][\lambda,\lambda']$ and this ideal equals the kernel of the map $\Z[1/2][\lambda,\lambda'] \rightarrow R$. Thus, we have a finite injective morphism $\Z[1/2][\lambda,\lambda']/I \hookrightarrow R$. Using the going up theorem, we get a surjection $\Spec R \rightarrow \Spec \Z[1/2][\lambda,\lambda']/I$ and the fact that the ring $\Z[1/2][\lambda,\lambda']/I $ is equidimensionnal of dimension two. Hence, the ideal $I$ has codimension one. 
 
 Moreover, the affine scheme $M'_{\Gamma(2)} \times M_{\Gamma(2)}'$ is isomorphic to $\mathbf{A}^2_{\Z[1/2]}$, hence it is a factorial scheme. The ideal $I$ is generated by an element $F(X,Y) \in \Z[1/2][X,Y]$ since the Picard group of a factorial ring is trivial. Thus, $V(I)$ is a principal Weil divisor. 
 
 Since the degree of the two projections $c, c \circ w_p : M_{\Gamma_0(p) \cap \Gamma(2)} \rightarrow M_{\Gamma(2)}$ is equal to $d=\# \mathbf{P}^1(\Z/p\Z)$ and $M_{\Gamma_0(p) \cap \Gamma(2)} \rightarrow M_{\Gamma(2)} \times M_{\Gamma(2)} $ is injective outside  CM-points and the special fiber at $p$, the degree of $F$ as a polynomial in $X$ equals to the degree of $F$ as a polynomial in $Y$, equals to $p+1$. Thus, we have $F(X,Y)= u X^d + v Y^d + \sum_{i+j \leq d \\, i < d, j < d} a_{i,j} X^i Y^j$, where $u,v$ are invertible in $\Z[1/2]$. Moreover, since $w_p$ is an involution, we have $F(X,Y)= \alpha \cdot F(Y,X)$ where $\alpha$ is invertible in $\mathbf{Z}[1/2]$. We must have $\alpha=1$ since else, we have $F(X,X)=0$. This is impossible since this implies that every elliptic curve over $\mathbf{C}$ has CM by some quadratic order. Thus we can assume that $F \in \mathbf{Z}[1/2][X,Y]$ is monic in $X$ and $Y$. It is the clear that $F = F_p$ (the $p$-th modular polynomial). Moreover, the coefficients of $F_p$ are in some cyclotomic ring and thus are in $\mathbf{Z}$. More precisely, the Fourier coefficients of $\lambda(p\tau)$ and $\lambda((\tau+a)/p)$ are in $\mathbf{Z}[\zeta_p]$. Consequently, any coefficient of $F_p(X,\lambda)$ (as a polynomial in $X$) is a polynomial in $\lambda$ with coefficients in $\mathbf{Z}[\zeta_p]$ (\cf \cite[Chapter 5 Theorem 2]{Lang_elliptic}).

We have thus proved the first part of the following result.
\begin{prop}\ We have $F_p \in \mathbf{Z}[X,Y]$ and $F_p(X,Y)=F_p(Y,X)$. The curve $M_{\Gamma_0(p) \cap \Gamma(2)} \otimes \Q$ is birational to $\Spec \Q[X,Y]/(F_p(X,Y))$.
\end{prop}

\begin{proof}

Since $M_{\Gamma_0(p) \cap \Gamma(2)})' $ is irreducible, $\Spec \Z[1/2][X,Y]/(F_p(X,Y))$ is irreducible (we have $\Z[1/2][X,Y]/(F_p(X,Y))\subset R$ and $R$ is an integral domain). Let $K(M_{\Gamma_0(p) \cap \Gamma(2)})$ be the field of rational functions of $M_{\Gamma_0(p) \cap \Gamma(2)}$ and $K(\Z[1/2][X,Y]/(F_p(X,Y))$ be the function field of $\Spec \Z[1/2][X,Y]/(F_p(X,Y))$.

We have inclusions
$$\Q(\lambda) \subset K(\Z[1/2][X,Y]/(F_p(X,Y))) \subset K(M_{\Gamma_0(p) \cap \Gamma(2)})$$ and by comparing degrees, we have $K(\Z[1/2][X,Y]/(F_p(X,Y))) = K(M_{\Gamma_0(p) \cap \Gamma(2)}) $. Thus, the curve $M_{\Gamma_0(p) \cap \Gamma(2)} \otimes \Q$ is birational to $\Spec \Q[X,Y]/(F_p(X,Y))$, since they have the same field of rational functions and $K(\Z[1/2][X,Y]/(F_p(X,Y))) \cap \bar{\Q}=\Q$ (the cusps of $M_{\Gamma_0(p) \cap \Gamma(2)}\otimes \Q$ are $\Q$-rational).
 
  \end{proof}

If $E$ is an elliptic curve over $\overline{\mathbf{F}}_p$ and $ E \rightarrow E^{(p)}$ be the Frobenius, then $E$ is ordinary if and only if the kernel of Frobenius is isomorphic to the finite flat group scheme $\mu_p$. The Atkin--Lehner involution $w_p$ sends the multiplicative component of the special fiber of $M_{\Gamma_0(p) \cap \Gamma(2)}$ to the \'etale component via $\lambda \mapsto \lambda^p$. 

\begin{cor}\label{Kronecker_congruence}
The reduction of $F_p(X,Y)$ modulo $p$ is $(X^p -Y) (X-Y^p)$.
\end{cor}

\begin{proof}
Let $x$ be an element of $\mathfrak{X}(\bar{k})$ corresponding to $(E,\alpha_2,H)$ such that $E$ is not supersingular. We have two cases: 

If $H$ is a multiplicative subgroup of order $p$, then from the discussion above, it is clear that $\lambda(x)^p=\lambda(w_p(x))$. 

Otherwise, $H$ is \'etale and $\lambda(x)=\lambda(w_p(w_p(x)))=\lambda(w_p(x))^p$. 

Moreover, since the open given by the complementary of supersingular elliptic curves is dense in the special fiber of $\mathfrak{X}$, the zeros of the polynomial $(X^p-Y)(Y^p-X) \in \overline{\mathbf{F}}_p[X,Y]$ are zeros of $F_p$ modulo $p$. Furthermore, $\mathbf{Z}[1/2][X,Y]/(F_p(X,Y))$ is reduced ($\Spec \mathbf{Z}[1/2][X,Y]/(F_p(X,Y))$ is the scheme theoretic image of $(c,c\circ w_p)$). Thus, in this ring we have $(X^p-Y)(Y^p-X) = 0$ and by comparing the degree, we have the equality.
\end{proof}

\begin{rem}
This corollary could be proved in a more down-to-earth way, like in \cite[Chapter $5$ \S 2]{Lang_elliptic}.
\end{rem}
\section{$p$-adic uniformization and the reduction map}

Let $\mathfrak{X}$ be the modular curve $M_{\Gamma(p) \cap \Gamma(2)} \otimes \cO_K$. Since the singularities of the special fiber of $\mathfrak{X}$ are $k$-points, Mumford's Theorem \cite{M} shows the existence of a free discrete subgroup $\Gamma \subset \PGL_2(K)$ (\textit{i.e.} a Schottky group) and of a $\Gal(\bar{K}/K)$-equivariant morphism of rigid spaces:$$\tau: \mathfrak{H}_{\Gamma} \rightarrow \mathfrak{X}_K^{an}$$ inducing an isomorphism $\mathfrak{X}_K^{an} \simeq \mathfrak{H}_{\Gamma}/\Gamma$, where $\mathfrak{H}_{\Gamma}=\mathbf{P}^1_K - \mathcal{L}$ and $\mathcal{L}$ is the set of limit points of $\Gamma$. Note that $\mathfrak{H}_{\Gamma}$ is an admissible open of the rigid projective line $\mathbf{P}^1_K$. 

Let $\mathcal{T}_{\Gamma}$ be the subtree of the Bruhat--Tits tree for $\PGL_2(K)$ generated by the axes whose ends correspond to the limit points of $\Gamma$. 
Mumford constructed in \cite{M} a continuous map $\rho: \mathfrak{H}_{\Gamma} \rightarrow \mathcal{T}_{\Gamma}$ called the reduction map.

The special fiber of $\mathfrak{X}$ has two components, and each component has $3$ cusps. One of these components, which we call the \textit{\'etale} component, classifies elliptic curves or $2p$-sided N\'eron polygons over $\bar{k}$ with an \'etale subgroup of  order $p$ and a basis of the $2$-torsion. The other component, which we call the \textit{multiplicative} component, classifies elliptic curves or $2$-sided N\'eron polygons over $\bar{k}$ with a multiplicative subgroup of order $p$ and a basis of the $2$-torsion. The involution $w_p$ sends a $2p$-gon to a $2$-gon. Let $c$ and $c'=w_p(c)$ be two cusps of $M_{\Gamma_0(p)\cap \Gamma(2)}(\C)$ such that $c$ is above $c_{\Gamma(2)}$ and we know also by proposition \cite[7.4]{B-L} that $c'$ is also above $c_{\Gamma(2)}$ ($\xi_{c}$ corresponds to a $2p$-gon and $\xi_{c'}=w_p(\xi_c)$ corresponds to a $2$-gon). 

The dual graph $\Delta$ of the special fiber of $\mathfrak{X}$ has two vertices $v_{c'}$ and $v_{c}$ indexed respectively by the cusps $\xi_{c'}$ and $\xi_c$. There are $g+1$ edges $e_i$ ($i \in \{0,...,g\}$) corresponding to supersingular elliptic curves with a $\Gamma(2)$-structure. We orient these edges so that they point out of $v_{c'}$.

The Atkin-Lehner involution $w_p$ exchanges the two vertices $v_{c'}$ and $v_{c}$ and also acts on edges (reversing the orientation). More precisely, if $E_i$ is a supersingular elliptic curve corresponding to $e_i$, then $w_p(e_i)=e_j$ where $e_j$ is the elliptic curve associated to $E_i^{(p)}=w_p(E_i)$ (here $w_p$ is the Frobenius). Thanks to Lemma \ref{normalization} below, one can identify the generators $\{\gamma_i\}_{1\leq i\leq g+1}$ of $\Gamma$ with $(e_i-e_0)_{1\leq i \leq g+1}$. 

Let $\tilde{v}_{c}$ and $\tilde{v}_{c'}$ be two neighbour vertices of $\mathcal{T}_{\Gamma}$ reducing to $v_{c'}$ and $v_{c}$ respectively, such that the edge linking $\tilde{v}_{c}$ to $\tilde{v}_{c'}$ reduces to $e_0$ modulo $\Gamma$. For $0 \leq i \leq g$, let $\tilde{e}_i'$ be an edge pointing out of $\tilde{v}_{c'}$ and reducing to $e_i$ modulo $\Gamma$. Let $\tilde{e}_i$ be oriented edges of $\mathcal{T}_{\Gamma}$ lifting $e_i$ and pointing to $\tilde{v}_c$. Note that $\tilde{e}_0 = \tilde{e}_0'$.

 Let $A=\rho^{-1}(\tilde{v}_{c})$ and $A'=\rho^{-1}(\tilde{v}_{c'})$.  $A$ (resp. $A'$) is the complement of $g+1$ open disks in $\mathbf{P}^1_K$, hence $\mathbf{P}^1_K - A= \underset{0\leq i \leq g}{\coprod} B_i$ and $\mathbf{P}^1_K - A'=\underset{0\leq i \leq g}{\coprod} C'_i$ where $0 \leq i \leq g$. We index $B_i$ and $C'_i$ such that $A \subset C'_0$, $A' \subset B_0$, $B_i$ and $C'_i$ are associated to $\tilde{e}_i'$ and $\tilde{e}_i$ respectively.

For all $0\leq i \leq g$, $\rho^{-1}(\tilde{e}_i')=c_i$ is an annulus of $C'_i$ and $C_i=C'_i-\rho^{-1}(\tilde{e}_i')$ is a closed disk; we also have $\mathbf{P}^1_K-C_0=B_0$. We have $$\mathbf{P}^1_K - \rho^{-1}(\underset{0 \leq i \leq g}{\cup} \tilde{e_i}' \cup \tilde{v}_{c'})=\underset{0\leq i \leq g}{\coprod} C_i \text{ .}$$

Note that $\tilde{v}_{c} \cup \tilde{v}_{c'} \cup_i \{\tilde{e}_i'\}$ is a fundamental domain of $\mathcal{T}_{\Gamma}$, so $$D=\mathbf{P}^1_K-\underset{1\leq i \leq g}{\coprod} B_i \cup \underset{1\leq i \leq g}{\coprod} C_i$$ is a fundamental domain of $\mathfrak{H}_{\Gamma}$.

\begin{lemma}\cite[Lemma 3.3]{B-L}\label{normalization}
We can choose $\Gamma$ such that there is a Schottky basis $\alpha_1,..., \alpha_g$ of $\Gamma$,  and a fundamental domain $D$ satisfying:

\begin{enumerate}

\item $B_i$ is the open residue disk in the closed unit disk of $\mathbf{P}^1_K$ which reduces to $\lambda(e_i)^p$, $\forall 0\leq i \leq g$.

\item For $1\leq i \leq g$, $\alpha_i$ corresponds, under the identification $\Gamma^{ab} = \Z[S]^0$, to $e_i-e_0$.

\item $\alpha_i$ sends bijectively $\mathbf{P}^1_K-B_i$ to $C_i$ and $\alpha_i^{-1}$ sends bijectively $\mathbf{P}^1_K-C_i$ to $B_i$.

\item The annulus $c_i$ is isomorphic, as a rigid analytic space, to $\{z , |p| < |z|< 1\}$.

\end{enumerate}

\end{lemma}

For $a, b \in \mathfrak{H}_{\Gamma}$, define the meromorphic function $\theta(a,b; z) = \theta((a)-(b); z)$ ($z \in \mathfrak{H}_{\Gamma}$) by the convergent product 
$$ \theta(a,b;z)= \prod_{\gamma \in \Gamma} \frac{z-\gamma a}{z-\gamma b} \text{ .}$$

See \cite{D} for the basic properties of these theta functions. 

For all $a,b \in \mathfrak{H}_{\Gamma}$, the theta series $\theta(a,b,.)$ converges and defines a rigid meromorphic function on $\mathfrak{H}_{\Gamma}$ (which is modified by a constant if we conjugate $\Gamma$). We extend $\theta$ to degree zero divisors $D$ of $\mathfrak{H}_\Gamma$. The series $\theta(D,.)$ is entire if and only if $\tau_{*}(D)=0$, where we recall that $\tau : \mathfrak{H}_{\Gamma} \rightarrow \mathfrak{X}^{an}_K$ is the uniformization. 

The proposition below follows from \cite{D} (see also \cite{shalit}).

\begin{prop}\label{theta}\cite{D}
\begin{enumerate}
\item $\theta(a,b,z)=c(a,b,\alpha) \theta(a,b,\alpha z)$, where $\alpha \in \Gamma$ and $c(a,b,\alpha \beta)=c(a,b,\alpha)c(a,b,\beta)$.
\item The function $u_{\alpha}(z)=\theta(a,\alpha a,z)$ does not depend on $a$, and $u_{\alpha \beta}=u_{\alpha} u_{\beta}$.
\item $c(a,b,\alpha)= u_{\alpha}(a)/u_{\alpha}(b)$.
\item $\theta(a,b,z)/\theta(a,b,z')=\theta(z,z',a)/\theta(z,z',b)$.

\end{enumerate}
\end{prop}

We recall that $\Phi:\mathbf{Z}[S]^0 \times \mathbf{Z}[S]^0 \rightarrow K^{\times}$ is defined by:

$$\Phi (\alpha, \beta)=\theta(a,\alpha a,z)/\theta(a,\alpha a, \beta z)=u_{\alpha}(z)/u_{\alpha}(\beta z).$$

The results of Mumford \cite{M} imply that we can identify $\Gamma^{ab}$ with $\Z[S]^0$. and that $\cT_{\Gamma}$ is the universal covering of the graph $\Delta$. Moreover, Manin and Drinfeld proved that $v_K \circ \Phi$ is positive definite ($v_K$ is the $p$-adic valuation of $K$). According to lemma \cite[4.2]{B-L}, the pairing $\Phi$ takes values in $\Q_p^{\times}$.

We recall that we defined in \cite{B-L} an extension $\Phi: \mathbf{Z}[S] \times \mathbf{Z}[S] \rightarrow K^{\times}$ as follow:

For all $0 \leq i \leq g$, we had chosen $\xi_c^{(i)}$ (resp. $\xi_{c'}^{(i)}$) in $\mathfrak{H}_{\Gamma}$ which reduces modulo $\Gamma$ to the cusp $\xi_c \otimes \Q_p$ (resp. $\xi_{c'} \otimes \Q_p$), and such that $\xi_c^{(i)}$ and $\xi_{c'}^{(i)}$ are separated by an annulus reducing to $e_i$.  Let $\tilde{v}_c^{(i)}$ and $\tilde{v}_{c'}^{(i)}$ be two neighbour vertices of $\mathcal{T}_{\Gamma}$ above $v_c$ and $v_{c'}$ respectively, separated by an edge reducing to $e_i$. We fix $\tilde{v}_c^{(0)} = \tilde{v}_c$ and $\tilde{v}_{c'}^{(0)} = \tilde{v}_{c'}$. Thus, we had chosen $\xi_c^{(i)}$ (resp. $\xi_{c'}^{(i)}$) in $\rho^{-1}(\tilde{v}_c^{(i)})$ (resp.  $\rho^{-1}(\tilde{v}_{c'}^{(i)})$). Let for all $0 \leq i \leq g$, $\xi_c^{(i)} = z_0 \in A$, then the $\xi_{c'}^{(i)}$ satisfy
$$\xi_{c'}^{(i)} = \alpha_i^{-1}(\xi_{c'}^{(0)}) \in B_i \text{ .}$$
Therefore, we have $\xi_{c'}^{(0)} \in \rho^{-1}(\tilde{v}_{c'})=A'$ and $\alpha_i^{-1}(A') \subset \alpha_i^{-1}(\mathbf{P}^1 - C_i') \subset B_i$. We can assume also without losing in generality that $z_0 \neq \infty$.

Thus, for any $a \in \mathfrak{H}_{\Gamma}$, we had defined an extension of $\Phi$ to a pairing on $\Z[S]^0 \times \Z[S]$ (and taking values in $K$) as follow : 

for all $\alpha \in \Gamma$, 
\begin{equation}
\Phi(\alpha, e_i) = \frac{\theta(a, \alpha(a), \xi_{c'}^{(i)})}{\theta(a, \alpha(a), \xi_{c}^{(i)})}=\frac{u_{\alpha}(\xi_{c'}^{(i)})}{u_{\alpha}(\xi_{c}^{(i)})}=\frac{u_{\alpha}(\xi_{c'}^{(i)})}{u_{\alpha}(z_0)}
\label{def_Q_1}
\end{equation}

Let $\lambda' :\mathfrak{X}_K\rightarrow \mathbf{P}^1_K$ be $\lambda' = \lambda \circ w_p$. The Atkin--Lehner involution acts on $\Gamma \backslash \mathcal{T}_{\Gamma}$ and lifts to an orientation reversing involution $w_p$ of $\mathcal{T}_{\Gamma}$ (by the universal covering property). By \cite{G} ch. VII Sect. $1$, there is a unique class in $N(\Gamma)/\Gamma$ (where $N(\Gamma)$ is the normalizer of $\Gamma$ in $\PGL_2(K)$) inducing $w_p$ on $\mathcal{T}_{\Gamma}$. We denote by $w_p$ the induced map of $\mathfrak{H}_{\Gamma}$ (it is only unique modulo $\Gamma$).

Fix $0 \leq i , j \leq g$. 

Let $z \in \mathfrak{H}_{\Gamma}$ near $\xi_{c}^{(i)}$ and $z'$ near $\xi_{c'}^{(i)}$ such that $\tau(z) = w_p(\tau(z'))$. Recall that by hypothesis, $\xi_{c}^{(i)}=z_0$ is independent of $i$.

For $0 \leq i,j \leq g$,  We bilinearly extend $\Phi$ to $\Z[S] \times \Z[S]$ in \cite{B-L} as follow:
\begin{equation}
\Phi(e_i,e_j) = \text{lim } \lambda'(\tau(z))^2 \cdot \frac{\theta(z',z,\xi_{c'}^{(j)})}{\theta(z',z,\xi_{c}^{(j)})}
\label{def_Q_2}
\end{equation}
where $z$ and $z'$ approach $\xi_{c}^{(i)}$ et $\xi_{c'}^{(i)}$ respectively.
Since at $z = z_0$, $\lambda' \circ \tau$ has a simple pole and the numerator and denominator have a simple zero and simple pole respectively, $\Phi(e_i,e_j)$ is finite, and is in $K^{\times}$ since $K$ is complete (we choose $z$ and $z'$ in $K$ to compute the limit).

\section{Proof of the main Theorem}

\subsection{Case where $\lambda(e_0) \in \mathbf{F}_p$.} Assume that $\lambda(e_0) \in \mathbf{F}_p$. We can choose a lift of the involution $w_p$  to $\tilde{w}_p$ of $N(\Gamma) \subset \mathrm{PGL}_2(K)$ preserving the edge $e_0$ and reversing the orientation of this edge. Thus, $w_p$ preserves the annulus $\tilde{e_0}'$, so sends $A$ to $A'$. Hence, $\tilde{w}_p$ is an involution (we have $\tilde{w}_p^2 \in \Gamma$ and the stabilizer of an edge in $\Gamma$ is trivial, so $\tilde{w}_p^2=1$).

Let $\zeta=w_p(\infty)$ (\textit{i.e.} $w_p(\zeta)=\infty$). Then our choice of the fundamental domain of $\mathfrak{H}_\Gamma$, implies that $B_0 = \{z \in \mathbf{P}^1_K, \mid z -\zeta \mid_p < 1\}$, $\mathbf{P}^1_K \backslash C'_0 = B_0 \backslash c_0 = \{z \in \mathbf{P}^1_K, \mid z-\zeta \mid_p < \mid p \mid_p\}$, and $B_0 \mod p = \lambda(e_0)$. 

Any involution of $\mathbf{P}^1_K$ exchanging $0$ and $\infty$ has the form $$z \rightarrow \frac{\pi}{z},$$
where $\pi$ is an uniformizer $\cO_K$. Thus, we have:

\begin{equation}\label{inverting}
w_p(z)-\zeta = \frac{\pi}{z-\zeta}
\end{equation}

\begin{lemma} 
We have $\Phi(e_0,e_0) \sim \pi$.
\end{lemma}

\begin{proof}
Recall the definition:
$$\Phi(e_0,e_0) =  \text{lim } \lambda'(\tau(z))^2 \cdot \frac{\theta(z',z,\xi_{c'}^{(0)})}{\theta(z',z,\xi_{c}^{(0)})}$$
where $z$ goes to $z_0 \in A$ and $z' = w_p(z)$. We now do a similar analysis as in \cite[Section $6.1$]{B-L}.
By \cite[Proposition $6.3$]{B-L}, 
$$ \text{lim } \frac{\lambda'(\tau(z))\cdot (z_0-z)}{z_0^2} = 1\text{ .}$$
Recall also that 
$$\frac{\theta(z',z,\xi_{c'}^{(0)})}{\theta(z',z,\xi_{c}^{(0)})} = \prod_{\gamma \in \Gamma} \frac{(z'-\gamma(\xi_{c'}^{(0)}))\cdot (z-\gamma(z_0))}{(z-\gamma(\xi_{c'}^{(0)}))\cdot (z'-\gamma(z_0))} \text{ .}$$
If $\gamma \neq 1$, then $\gamma(z_0)$and $\gamma(\xi_{c'}^{(0)}$ both lie in the same disk $B_i$ or $C_i$ for some $i$ which does not contain $z'$, $\xi_{c'}^{(0)}$ and $z_0$. Since $\mid z_0 \mid_p >1$ and $\mid z' \mid_p \leq 1$, the only term in this infinite product which is not a principal unit is the one corresponding to $\gamma=1$, which is equivalent modulo principal units to
$$\frac{(z'-\xi_{c'}^{(0)})\cdot (z-z_0)}{-z_0^2} \text{ .}$$
Thus, we have:
$$\Phi(e_0,e_0) \sim \frac{z_0^2}{z_0-z} \cdot (z'-\xi_{c'}^{(0)}) \text{ .}$$
To conclude the proof of the Lemma, note that 
$$z' - \xi_{c'}^{(0)} = (z' - \zeta) - (\xi_{c'}^{(0)}-\zeta) = \frac{\pi}{z-\zeta} - \frac{\pi}{z_0 - \zeta} \sim\frac{\pi \cdot (z_0-z)}{z_0^2} \text{ .}$$

\end{proof}

\subsubsection{Conclusion of the proof of point $(ii)$ Theorem \ref{main_theorem} in the case where $\lambda(e_0) \in \mathbf{F}_p$}

To conclude the proof of point $(ii)$ of Theorem \ref{main_theorem}, it remains to show that
$$\pi \sim F_p(\beta_0, \beta_0^p) $$
for any lift $\beta_0$ of $\lambda(e_i)$ in $K$.
 
The proof is really the same as \cite[3.1--3.3]{shalit2}, using our analogous fundamental domain for $\Gamma$, and replacing $j$ by $\lambda'$. Thus, we shall be really sketchy and refer the reader to de Shalit's paper for details.

We recall that have $ord_p(\pi) = 1$. By slight abuse of notation we shall denote $\lambda$ for $\lambda \circ c $ and $\lambda'$ for $\lambda \circ c\circ w_p$.

Let $y=z- \zeta $; it identifies the annulus $a=\rho^{-1}(\tilde{e_0})$ with $$A(p, 1):=\{ x \in \mathbf{P}^1_K, \mid  p\mid_p<\mid x \mid_p <1\} \text{ .}$$ Consider the map $\Psi : a \rightarrow B_0$ defined by 
\begin{equation}\label{map 1}
\Psi(z) =  \lambda' \circ \tau (z)
\end{equation}

This is a covering of $B_0$ by $a$ since $a$ is the intersection of our fundamental domain $D$ with $B_0$ (although it might seems surprising compared to the classical complex situation, such a covering indeed exists).

There exists $\beta_0 \in B_0$, such that 
\begin{equation}
\Psi(z)=\beta_0 + \sum_{n \geq 1 } a_n y^n + \sum_{n \geq 1} b_n (\pi/y)^n
\end{equation}

where all the coefficients  $a_n$, $b_n$ are in $K$ (since $\Psi$ is $K$-rational by Proposition \ref{iso lambda}). 
For $y$ 

Using \cite[Lemma 6.2]{B-L} and similar computations as \cite[p. 143-144]{shalit2}, we get:
\begin{lemma}\label{coefficients}
We have
$a_1 \sim 1$, $ b_p \sim 1$ and for $n<p$, $\mid b_n \mid_p < \mid p \mid_p$.
\end{lemma}

For $t \in \mathbf{C}_p^{\times}$ such that $\mid p \mid_p < \mid t \mid_p < 1$, let $a(t,1)$ be the open annulus where $\mid t \mid_p <\mid y \mid_p < 1$.

For $t$ close enough to $1$ and $y \in a(t,1)$, we set
$$u = \lambda' - \lambda^p = \Psi(\pi/y) - \Psi(y)^p \text{ .}$$
We have, by definition:
$$ F_p(\lambda, \lambda^p + u) = 0 \text{ .}$$
This gives us, using partial derivatives and Corollary \ref{Kronecker_congruence}:
$$u\cdot (\lambda^{p^2}-\lambda+u^p) = -pR(\lambda)-p\cdot  h(u,\lambda)$$
where $R(X) = \frac{F_p(X,X^p)}{p} \in \mathbf{Z}[X]$ and $h(X,Y) \in \mathbf{Z}[X,Y]$ is some integer coefficients polynomial. 

We work modulo the ideal $I(t)$ generated by rigid analytic functions on $a(t,1)$ which are strictly smaller than $\mid p \mid_p$ in absolute value. The term $-pR(\lambda) - p\cdot h(u,\lambda)$ is congruent to $-pR(\lambda)$ modulo $I(t)$. A simple computation using Lemma \ref{coefficients} shows that we must have $u \equiv \pi/y$ modulo $I(t)$ and 
$$\lambda^{p^2}-\lambda = \Psi(y)^{p^2}-\Psi(y) \equiv -y + O(y^2) \text{ (modulo }p  \text{).}$$

This shows what we needed to conclude the proof of point $(ii)$ of Theorem \ref{main_theorem}:
$$(\pi/y) \cdot y \sim pR(\lambda)\text{ .}$$

\subsubsection{Existence of CM lifts}
In this section, we prove part $(iii)$ of Theorem \ref{main_theorem} in the case $\lambda(e_0) \in \mathbf{F}_p$. 
\begin{prop}\label{CM_lift}
Let $\lambda \in \mathbf{F}_p$ be a supersingular $\lambda$-invariant. Then $p\equiv 3 \text{ (modulo }4\text{)}$ and there exists precisely two $\lambda$-invariants in $\mathbf{Q}_p(\sqrt{-p})$ lifting $\lambda$, with complex multiplication by $\mathbf{Z}[\frac{1+\sqrt{-p}}{2}]$.

Furthermore, when $p\equiv 3 \text{ (modulo }4\text{)}$ (and $p\geq 5$ as usual), the number of supersingular $\lambda$-invariants in $\mathbf{F}_p$ is $3\cdot h(-p)$ where $h(-p)$ is the class number of $\mathbf{Q}(\sqrt{-p})$.
\end{prop}
\begin{proof}

Let $\mathcal{O} = \mathbf{Z}[\frac{1+\sqrt{-p}}{2}]$ and $F$ be the fraction field of $\mathcal{O}$. Let $\mathfrak{a}$ an element of the ideal class group of $\mathcal{O}$. We denote by $j(\mathfrak{a})$ the $j$-invariant of the isomorphism class of the elliptic curve $\mathbf{C}/\mathfrak{a}$. It is classical (\cf for instance \cite{Silverman} Theorem $5.6$) that if $\lambda$ is any $\lambda$-invariant above $j(\mathfrak{a})$, then $F(\lambda)$ is an extension of $F(j)$ contained in the ray class field of $F$ of conductor $2\cdot \mathcal{O}$. 
\begin{lemma}\label{split_lambda}
The ideal above $p$ in $\mathcal{O}$ is totally split in $F(\lambda)$.
\end{lemma}
\begin{proof}
If $p \equiv -1 \text{ (modulo } 8\text{)}$, then $2$ splits in $F$, so $\sqrt{-p}-1 \in \mathfrak{p}_2$ if $\mathfrak{p}_2$ is any prime ideal of $\mathcal{O}$ above $(2)$. Thus, by class field theory, since $(\sqrt{-p})$ is principal, it splits in the ray class field of conductor $2\cdot \mathcal{O}$ and we are done.

If $p \equiv 3 \text{ (modulo } 8\text{)}$, then $2$ is inert in $F$. The prime ideal above $2$ in $\mathcal{O}$ is $\mathfrak{P}_2 = (2, \alpha^2+\alpha+1)$ where $\alpha = \frac{1+\sqrt{-p}}{2}$. Since $\alpha^2=\alpha - \frac{p+1}{4}$, we have $\mathfrak{P}_2 = (2, \frac{3-p}{4}+\sqrt{-p})$.  Thus we have $\sqrt{-p}-1 \in \mathfrak{P}_2$. As above, class field theory shows that $(\sqrt{-p})$ splits in the ray class field of $F$ of conductor $\mathfrak{P}_2$, which concludes the proof of the lemma.
\end{proof}

\begin{lemma}
Let $\lambda \in \overline{\mathbf{Z}}_p$ such that the Legendre curve $E_{\lambda} : y^2=x(x-1)(x-\lambda)$ has supersingular reduction. Then $\lambda$ is a root of $F_p(X,X)$ if and only if $E_{\lambda}$ has CM by $\mathbf{Z}[\frac{1+\sqrt{-p}}{2}]$. Furthermore in this case $\lambda$ is a simple root of $F_p(X,X)$.
\end{lemma}
\begin{proof}
It is clear that $E_{\lambda}$ has CM by a quadratic order $\mathcal{O}$ such that $p$ either splits or ramifies in the fraction field. But $p$ has to ramify since the reduction of $E_{\lambda}$ is supersingular (this comes from the standard description of the local galois representation attached to a supersingular elliptic curve). Furthermore, there is an endomorphism of $E$ whose square is $-p$. Thus, $\sqrt{-p} \in \mathcal{O}$. But in fact we have $\sqrt{-p} \in 1+2\mathcal{O}$ since the endomorphism $\sqrt{-p}$ has to preserve the $\Gamma(2)$-structure. Thus we have $\mathcal{O} = \mathbf{Z}[\frac{1+\sqrt{-p}}{2}]$.
\end{proof}

Corollary \ref{Kronecker_congruence} shows that
$$F_p(X,X)\equiv -(X^p-X)^2 \text{ (modulo }p\text{).}$$
Thus, any supersingular $\lambda$-invariant in $\mathbf{F}_p$ is a double root of $F_p(X,X)$. Using the previous lemma, we get:

$$\prod_{\lambda \text{ supersingular in }\mathbf{F}_p}(X-\lambda)^2 \equiv \prod_{[\mathfrak{a}] \in Cl(\mathbf{Z}[\frac{1+\sqrt{-p}}{2}]) \text{, }\lambda\text{ such that } j(\lambda)=j(\mathfrak{a})} (X-\lambda) \text{ .}$$
This shows that for any supersingular $\lambda$-invariant in $\mathbf{F}_p$, $\lambda$ has two CM lifts in characteristic $0$ which have CM by $\mathbf{Z}[\frac{1+\sqrt{-p}}{2}]$, and by Lemma \ref{split_lambda}, these lifts can be seen as living in $\mathbf{Q}_p(\sqrt{-p})$. This formula also shows the last assertion of the Proposition on the number of supersingular $\lambda$-invariants in $\mathbf{F}_p$ (there are $6$ $\lambda$-invariants above each $j$-invariant since $j(\mathfrak{a}) \neq 0, 1728$ because $p \geq 5$). 
\end{proof}

We now finish the proof of point $(iii)$ of \ref{main_theorem}. This is done in a similar way as \cite{shalit2} p. 146. Let $\lambda_1$ and $\lambda_2$ be the two CM values of lambda invariants in $\mathbf{Q}_p(\sqrt{-p})$ which lift $\lambda(e_i)$ (which exist by Proposition \ref{CM_lift}) . It is clear that $\lambda_1$ and $\lambda_2$ are not in $\mathbf{Q}_p$, so they must be conjugate. Write $\lambda_1 = a+b\sqrt{-p}$ and $\lambda_2 = a-b\sqrt{-p}$ for some $a,b \in \mathbf{Z}_p$. By point $(ii)$ of \ref{main_theorem}, it suffices to prove:
$$ F_p(\lambda_1, \lambda_1^p) \equiv (\lambda_1-\lambda_1^p)^2 \text{ (modulo } p\sqrt{-p}\text{).} $$
We know that $F_p(\lambda_1,\lambda_1)=0$. Therefore, we have:
$$0 = F_p(\lambda_1, \lambda_1) = F_p(\lambda_1, \lambda_1^p+(\lambda_1-\lambda_1^p)) \equiv F_p(\lambda_1, \lambda_1^p) + (\lambda_1-\lambda_1^p)\cdot (\lambda_1^{p^2}-\lambda_1) \text{ (modulo }p\sqrt{-p}\text{)} $$
where the last congruence follows from Corollary \ref{Kronecker_congruence} (which gives $\partial_Y F_p(X,Y) \equiv -(X-Y^p) \text{ (modulo } p \text{)}$).
Since $\lambda_1^{p^2} \equiv \lambda_1^p \equiv a \text{ (modulo }p\text{)}$ and $\lambda_1 \equiv \lambda_1^p \text{ (modulo }\sqrt{-p}\text{)}$, we get:
$$F_p(\lambda_1, \lambda_1^p) \equiv (b\cdot \sqrt{-p})^2 \text{ (modulo }p\sqrt{-p}\text{)} $$
which concludes the proof of Theorem \ref{main_theorem} if $\lambda(e_0) \in \mathbf{F}_p$.

\subsection{Case $\lambda(e_0) \in \mathbf{F}_{p^2} \backslash \mathbf{F}_p$}

Assume now that $\lambda(e_0) \in \mathbf{F}_{p^2} \backslash \mathbf{F}_p$, and without loss of generality that $\lambda(e_0)^p=\lambda(e_g)$. In this case, we choose $w_p$ such that $w_p(\tilde{e}_0') = \tilde{e}_g'$. Since $w_p^2 \in \Gamma$, we have $w_p^2 = \alpha_g$ (see \cite[p. 14]{B-L} for more details).

Let $z_g^+ \in C_g$ (resp. $z_g^- \in B_g$) be the attractive (resp. repulsive) fixed point of $\alpha_g$. As in the case $\lambda(e_0) \in \mathbf{F}_p$, the idea is to compute $w_p$. Let
$$\sigma(z) = \frac{z-z_g^+}{z-z_g^-} \text{ .}$$
Then $\sigma \circ \alpha_g \circ \sigma^{-1}$ fixes $0$ and $\infty$, and we get
$$\sigma\circ w_p \circ \sigma^{-1} = \kappa \cdot z$$
for some $\kappa \in \mathbf{C}_p$ of absolue value $\mid p \mid_p$. We let 
$$\pi := -\kappa \cdot (z_g^+ - z_g) \text{.}$$

Similar arguments as in the case $\lambda(e_0) \in \mathbf{F}_p$ give:
\begin{lemma}
We have $$\Phi(e_0,e_0) \sim \pi$$ and 
$$\pi \sim F_p(\beta_0, \beta_0^p)$$
for any lift $\beta_0$ of $\lambda(e_0)$ in $K$.
\end{lemma}

We refer as before to \cite[Sections $4.1-4.2$]{shalit2} for details in the $j$-invariant case.

\end{document}